\documentclass[10pt,reqno,a4paper]{amsart}

\usepackage{amsfonts,amssymb}
\usepackage{rotating,pgf,tikz}

\title{A toy Neumann analogue of the nodal line conjecture}

\author{J. B.~Kennedy}

\dedicatory{\upshape
Group of Mathematical Physics, University of Lisbon\\
Campo Grande, Edif\'icio C6, P-1749-016 Lisboa, Portugal\\[.5em]
\texttt{jbkennedy@fc.ul.pt}
}

\newtheorem{theorem}{Theorem}[section]

\newtheorem{principle}[theorem]{Principle}
\newtheorem{problem}[theorem]{Problem}
\newtheorem{assumption}[theorem]{Assumption}

\theoremstyle{remark}
\newtheorem{remark}[theorem]{Remark}

\newtheorem{example}[theorem]{Example}

\numberwithin{equation}{section}
\numberwithin{figure}{section}

\newcommand{\R}{\mathbb{R}}

\begin{document}

\begin{abstract}
We introduce an analogue of Payne's nodal line conjecture, which asserts that the nodal (zero) set of any eigenfunction associated with the second eigenvalue of the Dirichlet Laplacian on a bounded planar domain should reach the boundary of the domain. The assertion here is that any eigenfunction associated with the first nontrivial eigenvalue of the Neumann Laplacian on a domain $\Omega$ with rotational symmetry of order two (i.e., $x\in\Omega$ iff $-x\in\Omega$) ``should normally'' be rotationally antisymmetric. We give both positive and negative results which highlight the heuristic similarity of this assertion to the nodal line conjecture, while demonstrating that the extra structure of the problem makes it easier to obtain stronger statements: it is true for all simply connected planar domains, while there is a counterexample domain homeomorphic to a disk with two holes.
\end{abstract}



\thanks{The work of the author was supported by the Funda{\c{c}}{\~a}o para a Ci{\^e}ncia e a Tecnologia, Portugal, via the program ``Investigador FCT'', reference IF/01461/2015, and project PTDC/MAT-CAL/4334/2014.}


\maketitle

\section{Introduction}
\label{sec:intro}

Let $\Omega \subset \R^d$, $d\geq 2$, be a bounded domain (connected, open set) with sufficiently regular boundary $\partial\Omega$ and denote by
\begin{displaymath}
	0 = \mu_1 (\Omega) < \mu_2 (\Omega) \leq \mu_3 (\Omega) \leq \ldots
\end{displaymath}
the eigenvalues of the Laplacian with Neumann boundary conditions
\begin{displaymath}
\begin{aligned}
	-\Delta \psi &= \mu \psi \qquad &&\text{in } \Omega,\\
	\frac{\partial \psi}{\partial \nu} &=0 &&\text{on }\partial\Omega,
\end{aligned}
\end{displaymath}
where $\nu$ is the outward-pointing unit normal to $\partial\Omega$. We will also write $\psi_k \in H^1 (\Omega)$ for any (real-valued) eigenfunction associated with $\mu_k = \mu_k (\Omega)$, always understood in the weak, i.e.~variational, sense.

Recall that by a classical theorem of Courant, for any $k \geq 1$ the zero, or nodal, set of $\psi_k$, i.e.\ $\overline{\{ x \in \Omega : \psi_k (x) = 0\}}$, divides $\Omega$ into at most $k$ connected components, called the nodal domains of $\psi_k$. In particular, for $k=2$, the two nonempty and connected sets
\begin{displaymath}
	\Omega^+ := \{ x \in \Omega : \psi_2 (x) > 0 \}, \qquad \Omega^- := \{ x \in \Omega : \psi_2 (x) < 0\}
\end{displaymath}
are the two nodal domains of $\Omega$. The same is true of the eigenvalues and eigenfunctions of the corresponding Laplacian with Dirichlet boundary conditions, $\psi=0$ on $\partial\Omega$, which we shall denote by $\lambda_k=\lambda_k(\Omega)$ and $\varphi_k \in H^1_0 (\Omega)$, $k\geq 1$, respectively.

Our starting point is the \emph{nodal line conjecture} formulated and popularised by Payne \cite[Conjecture~5, p.~467]{payne:67}, \cite{payne:73}, now more than 50 years old, which postulated that the nodal set of $\varphi_2$ must touch $\partial\Omega$: no nodal domain should be entirely contained in $\Omega$. As Payne notes in \cite{payne:67}, in the Neumann case this is true; one expects it in the Dirichlet case by way of analogy and the principle that the nodal domains represent a $2$-partition of $\Omega$ minimising the spectral energy, namely
\begin{displaymath}
	\lambda_2 (\Omega) = \inf \Big\{ \max \{\lambda_1 (\Omega_1),\lambda_1(\Omega_2)\}: \Omega_1,\Omega_2\subset 
	\Omega \text{ open, } \Omega_1 \cap \Omega_2 = \emptyset,\,\overline{\Omega_1\cup\Omega_2}=\overline\Omega 
	\Big\},
\end{displaymath}
with equality exactly when $\Omega_1$ and $\Omega_2$ are nodal domains of some $\varphi_2$ (there is a corresponding statement for $\mu_2$ and $\psi_2$). This problem is thus closely related to, and partly of interest because it provides a link to, the ``distribution'' of the nodal domains cum optimal $2$-partition: it should be suboptimal from the point of view of energy minimisation to have one nodal domain concentrated somewhere in the ``middle'' of $\Omega$, with the other occupying its ``periphery''.

In addition to being true for convex planar domains \cite{alessandrini:94,melas:92}, the nodal line conjecture is also known to hold on some classes of symmetric domains, and some long, thin ones (we refer to \cite{grebenkov:13,kennedy:13} for more, and more precise, references). On the other hand, there are known counterexamples \cite{fournais:01,hoffmann:97}, which may be chosen simply connected in dimension three or above \cite{kennedy:13} but which seem intrinsically (topologically and geometrically) complicated in the plane. Thus the conjecture is not universally valid; it depends on the influence of the geometry of the domain on properties of the eigenfunctions.

Still not all that much is known; in particular, there seems to be a paucity of general positive results: is the conjecture true for simply connected planar domains? For convex domains in higher dimensions? We know of essentially only one ``type'' of counterexample: can one formulate a general principle by which a domain should fail to satisfy the conjecture? Yet in recent years there has been little progress, possibly owing to our lack of tools for connecting the eigenfunctions to the geometry of the domain; and attention has shifted to the more fertile problem of spectral minimal partitions (see, e.g., \cite{helffer:10,helffer:15} for surveys of the latter and \cite[Sec.~5]{grebenkov:13} for a summary of more properties of and techniques related to nodal lines in general).

The purpose of this note is to remark on a simple problem for the Neumann Laplacian which seems to bear considerable similarity to the nodal line problem, but which appears to be more structured and thus far easier and more tractable to handle---a kind of ``toy'' problem of a similar flavour to the nodal line conjecture. As with the latter, it asks how the two nodal domains of $\psi_2$ are distributed throughout $\Omega$. To introduce the problem, we first need to restrict to a class of symmetric domains: for the rest of the paper, unless otherwise stated, we shall make the following assumption.

\begin{assumption}
\label{assumption}
The planar bounded, Lipschitz domain $\Omega \subset \R^2$ has $\mathbb{Z}_2$ rotational symmetry,\footnote{Much of what we shall do is immediately generalisable to higher dimensions, but restricting ourselves to the planar case will keep the exposition simple and improve a number of statements; this case is also the one of principal interest for the original nodal line problem. The symmetry, however, seems essential to the formulation of the problem.} that is, $x \in \Omega$ if and only if $-x \in \Omega$.
\end{assumption}

An immediate consequence of this is that the eigenfunctions $\psi_k$ may be chosen to reflect this symmetry: we may assume for each $k \geq 1$ that $\psi_k$ is either \emph{symmetric} (or \emph{even}), i.e., $\psi_k (x) = \psi_k (-x)$ for all $x \in \overline{\Omega}$, or \emph{antisymmetric} (or \emph{odd}), $\psi_k (x) = -\psi_k (-x)$. (In the case of a simple eigenvalue, the eigenfunction must be one or the other. In the case of higher multiplicity, we may choose a basis of eigenfunctions each of which is either even or odd.) Since $\psi_1$ is always constant, it is trivially symmetric.

For $\psi_2$, the natural thing to \emph{expect} 
is:

\begin{principle}
\label{principle}
Any eigenfunction associated with $\mu_2 (\Omega)$ should be antisymmetric.
\end{principle}

The heuristic argumentation behind this is that, as with the principle behind the nodal line conjecture, the two nodal domains of any antisymmetric eigenfunction $\psi_2$ should divide $\Omega$ into two equal pieces, concentrated in opposite halves of $\Omega$, while if $\psi_2$ is symmetric, one nodal domain will be concentrated in the centre of $\Omega$, and the other at its periphery. We will give a few basic explicit examples of this, such as disks and rectangles, in Section~\ref{sec:examples}. We also note that reflection symmetry does not seem to yield a ``good'' problem; see Remark~\ref{rem:symmetry}, and there does not seem to be a natural generalisation of Principle~\ref{principle} to less symmetric domains. To the best of our knowledge, there has also not yet been a systematic investigation of the impact of symmetry of $\Omega$ in general on the nodal structure of $\psi_2$ or $\varphi_2$.

Let us now give a couple of results which demonstrate in a more formal way how this symmetric-antisymmetric principle mirrors the nodal line conjecture, and that it is easier to obtain stronger statements (both positive and negative) for it. It is thus to be hoped that studying this toy problem may, at the very least, yield insight into the nodal line conjecture and/or the distribution of the nodal domains $\Omega^\pm$ in $\Omega$.

We recall that nodal line conjecture is known to be true for convex planar domains, as cited above, and conjectured to be true for simply connected planar domains \cite{freitas:08,kennedy:13}, while the number of holes the counterexample planar domains need to have is completely unknown; there is not even a clear upper bound \cite{fournais:01}.

\begin{theorem}
\label{thm:simply-connected}
Suppose $\Omega$, in addition to satisfying Assumption~\ref{assumption}, is simply connected. Then any eigenfunction associated with $\mu_2 (\Omega)$ is antisymmetric.
\end{theorem}

\begin{theorem}
\label{thm:counterexample}
There exists a domain $\Omega$ satisfying Assumption~\ref{assumption}, which is homeomorphic to a disk with two holes, on which any eigenfunction associated with $\mu_2 (\Omega)$ is symmetric.
\end{theorem}

The proofs of our statements 
will be deferred until Section~\ref{sec:proofs}. The positive result is a direct and easy topological consequence of the well-known impossibility of $\psi_2$ having an interior nodal domain; the proof of the negative result is more involved (cf.~also Remark~\ref{rem:dirichlet}). Our choice of $\Omega$ in Theorem~\ref{thm:counterexample} is inspired by the counterexample given in \cite{hoffmann:97}, but much simpler: imagine a wheel consisting of a hub, a tire, and exactly two spokes connecting them (see Figure~\ref{fig:counterexample}). In this case, it is better for one nodal domain to concentrate in the hub, with the other on the tire, than to split both down the middle. With rather more effort, it should be possible to show that Theorem~\ref{thm:simply-connected} continues to hold for any doubly connected planar domain, thus providing a complete answer to the question of the effect of topology on Principle~\ref{principle} in the plane; see Remark~\ref{rem:doubly-connected}. Let us formulate this question explicitly.

\begin{problem}
\label{problem:doubly-connected}
Suppose $\Omega$ satisfies Assumption~\ref{assumption} and is homeomorphic to an annulus. Prove that any eigenfunction associated with $\mu_2 (\Omega)$ is antisymmetric.
\end{problem}

We can also ask an analogue of the other major open problem related to the nodal line conjecture: to prove that the two nodal domains of $\varphi_2$ on a \emph{convex} domain in higher dimensions both touch the boundary (see, e.g., \cite{kennedy:13}).

\begin{problem}
\label{problem:convex}
Suppose $\Omega \subset \R^d$, $d \geq 3$, is convex and satisfies $x \in \Omega$ if and only if $-x \in \Omega$. Prove that any eigenfunction associated with $\mu_2 (\Omega)$ is antisymmetric.
\end{problem}

Our example from Theorem~\ref{thm:counterexample} should be easily generalisable to higher dimensions.

\begin{remark}[Other boundary conditions]
\label{rem:dirichlet}
One could equally ask when Principle~\ref{principle} holds in the case of Dirichlet boundary conditions, i.e.~for $\lambda_2 (\Omega)$, or indeed for the second eigenvalue of the Laplacian with Robin boundary conditions $\frac{\partial \psi}{\partial\nu} + \alpha \psi = 0$ on $\partial\Omega$, say, for a constant $\alpha>0$. Both lead to well-posed problems, and we expect the statements (although not the proofs) to be robust to this choice. The trade-off is that negative results become easier to obtain, and positive ones harder, in the Dirichlet case: the Neumann condition seems to offer a better balance. The counterexample of Theorem~\ref{thm:counterexample} should continue to work with an easier proof, although we will not go into details. On the other hand, we do not know of a proof of Theorem~\ref{thm:simply-connected} without an additional convexity assumption on $\Omega$; indeed, the assertion becomes equivalent to the nodal line conjecture for simply connected planar domains satisfying Assumption~\ref{assumption}.
\end{remark}


\section{Some Examples}
\label{sec:examples}

Let us give a couple of simple explicit examples that illustrate why one expects $\psi_2$ to be antisymmetric, and how this corresponds to the intuition behind the nodal line conjecture. Here and throughout we will write $x = (x_1,x_2)=(r,\theta) \in \mathbb{R}^2$ for a (nonzero) point in the plane.

\begin{example}
Let $\Omega = B(0,1)$, the ball of unit radius. As is well known, every non-constant Neumann eigenfunction may be chosen to have the form $\psi (x) = J_k (\sqrt{\mu}r)e^{ik\theta}$ for some $k \in \mathbb{Z}$, where $J_k$ is the Bessel function of the first kind of order $k$ and $\sqrt{\mu} = j_{k,m}'$, the $m$th zero of the derivative of $J_k$ for some $m \geq 1$. For the first few eigenfunctions (chosen to be real) we obtain the nodal patterns shown in Figure~\ref{fig:disk}, corresponding to $J_{\pm 1}(j_{1,1}'r)e^{\pm i\theta}$ for $\mu_2=\mu_3=(j_{1,1}')^2$, $J_{\pm 2}(j_{2,1}' r)e^{\pm 2i\theta}$ for $\mu_4 = \mu_5 = (j_{2,1}')^2$ and $J_0 (j_{0,2}' r)$ for $\mu_6=(j_{0,2}')^2$ (under the convention $j_{0,1}'=0$; see, e.g., \cite[Sec.~9.5]{abramowitz:72}).
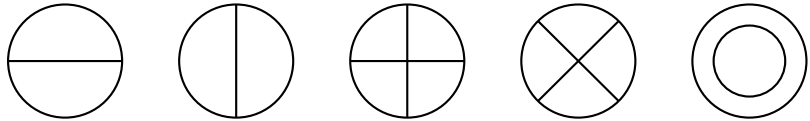
\begin{figure}[h]
\label{fig:disk}
\begin{center}
\begin{tikzpicture}[scale=1.5]
\draw[thick] (-3,0) circle [radius=0.5];
\draw[thick] (-3.5,0) -- (-2.5,0);
\draw[thick] (-1.5,0) circle [radius=0.5];
\draw[thick] (-1.5,0.5) -- (-1.5,-0.5);
\draw[thick] (0,0) circle [radius=0.5];
\draw[thick] (-0.5,0) -- (0.5,0);
\draw[thick] (0,0.5) -- (0,-0.5);
\draw[thick] (1.5,0) circle [radius=0.5];
\draw[thick] (1.5-0.35355,0+0.35355) -- (1.5+0.35355,0-0.35355);
\draw[thick] (1.5+0.35355,0+0.35355) -- (1.5-0.35355,0-0.35355);
\draw[thick] (3,0) circle [radius=0.5];
\draw[thick] (3,0) circle [radius=0.3138];
\end{tikzpicture}
\caption{Nodal patterns for canonically chosen real Neumann eigenfunctions of the disk for $\mu_2,\ldots,\mu_6$, respectively. In each case the nodal set is marked by the black line(s) through the disk.}
\end{center}
\end{figure}
We see that $\mu_2 = \mu_3$ have antisymmetric eigenfunctions, for any choice of basis elements, while $\mu_4 = \mu_5$ and $\mu_6$ have symmetric eigenfunctions. 
In fact only the one for $\mu_6$ has a nodal domain concentrated in the middle of $\Omega$, but the ones for $\mu_4$ and $\mu_5$ have four nodal domains. So when it comes to a comparison between all those eigenfunctions having two (the only ``candidates'' for $\psi_2$ in accordance with Courant's theorem) the first symmetric one corresponds~to~$\mu_6$.
\end{example}

\begin{example}
Fix numbers $a,b>0$ and let $\Omega = (-a,a) \times (-b,b)$, a rectangle centred at the origin. If $a>b$ then the first nontrivial eigenfunction is, up to scalar multiples,
\begin{displaymath}
	\psi_2(x_1,x_2)= \sin\left(\frac{\pi x_1}{2a}\right),
\end{displaymath}
which is in particular always antisymmetric. In the case of the square $a=b$ we obtain a space with multiplicity two, the function $\psi_3(x_1,x_2) = \sin(\pi x_2/2b)$ completing a basis. Then, again, every linear combination of $\psi_2$ and $\psi_3$ is antisymmetric. The numbering of the first eigenvalue with a symmetric eigenfunction, which has as a corresponding eigenfunction either $\psi(x) = \cos (\pi x_1/a)$ or $\sin(\pi x_1/2a)\sin(\pi x_2/2b)$,  depends on the ratio of $a$ to $b$: in particular, one may have arbitrarily many antisymmetric eigenfunctions before the first symmetric one.
\end{example}

\begin{remark}
\label{rem:symmetry}
The above example highlights why \emph{reflection} symmetry with respect to a fixed axis, for example, would not work in place of rotational symmetry: on the rectangle $\Omega = (-a,a) \times (-b,b)$ with $a>b$, $\psi_2$ is antisymmetric with respect to reflection in $\{x_2=0\}$ but symmetric in $\{x_1=0\}$. The rotational symmetry seems to circumvent this problem.
\end{remark}

\begin{example}
If $\Omega$ is a dumbbell (two disks or other similar domains symmetric to each other, connected by a thin ``handle'' through the origin), then the nodal line of $\psi_2$ will cut through the handle, in particular corresponding to an antisymmetric eigenfunction. This is an easy special case of Theorem~\ref{thm:simply-connected}. The same is true of long, thin domains.
\end{example}

\section{Proof of the statements}
\label{sec:proofs}

\begin{proof}[Proof of Theorem~\ref{thm:simply-connected}]
Suppose for a contradiction that $\psi_2$ is a symmetric eigenfuntion associated with $\mu_2$. Let $z,-z \in \partial\Omega$ be any two ``antipodal'' points on the boundary such that $\psi_2(z)=\psi_2(-z) >0$, say (if $\Omega$ is only Lipschitz, this may be understood in the sense of traces, recalling also that $\psi_2$ is still analytic inside $\Omega$). Noting that $\partial\Omega$ is connected, we distinguish between two cases: either $\psi_2>0$ on $\partial \Omega$ or $\psi_2$ changes sign on $\partial\Omega$.

\emph{Case 1:} $\psi_2\geq 0$ on $\partial\Omega$. In this case, the nodal domain $\Omega^-$ is entirely contained in $\Omega$: this leads to a contradiction via the standard argument $\mu_2 (\Omega) = \lambda_1 (\Omega^-) > \lambda_1 (\Omega) > \mu_2 (\Omega)$ using the characterisation of $\psi_2$ as the first Dirichlet eigenfunction on each of its nodal domains, the monotonicity of $\lambda_1$ with respect to domain inclusion, and a famous inequality originally due to P\'olya and Szeg\H{o}.

\emph{Case 2:} there exists some $y \in \partial \Omega$ such that $\psi_2(y)<0$. In this case, also $-y\in\partial\Omega$ and $\psi_2(-y)<0$. Note that $\partial\Omega \setminus \{z,-z\}$ consists of exactly two connected components, and $y$ and $-y$ lie in different components. In particular, there are at least four points on the boundary at which $\psi_2 = 0$, with $\psi_2$ changing sign in a neighbourhood of each. A simple topological argument based on the Jordan curve theorem shows that the nodal set of $\psi_2$ must now divide $\Omega$ into at least three nodal domains, a contradiction to Courant's theorem.
\end{proof}

\begin{remark}
\label{rem:doubly-connected}
If $\Omega$ is diffeomorphic to an annulus, then a similar topological argument, together with the fact that $\psi_2$ cannot have interior nodal domains, leads to the conclusion that, if $\psi_2$ is symmetric, then its nodal line must form a closed ring around the hole of $\Omega$, i.e., it divides $\Omega$ into an outer annulus (say, where $\psi_2>0$) and an inner annulus (where $\psi_2<0$). This is to be compared with the antisymmetric situation where the nodal line cuts $\Omega$ transversally, creating two half-doughnuts as nodal domains. Problem~\ref{problem:doubly-connected} consists in ruling out the former case.
\end{remark}

\begin{proof}[Proof of Theorem~\ref{thm:counterexample}]
We fix any numbers $0< r_1 < r_2 < r_3$ and set $B:=B(0,r_1) = \{x \in \R^2: |x|<r_1 \}$ and $A:= A(0,r_2,r_3) = \{x \in \R^2: r_2 < |x| < r_3 \}$.

For given $\varepsilon>0$ small, we form $\Omega_\varepsilon$ by uniting $B$, $A$ and two thin ``passages'' of angular width $\varepsilon$ along the $x_1$-axis: $\Omega_\varepsilon = A \cup B \cup U_\varepsilon^{+} \cup U_\varepsilon^{-}$, where
\begin{equation}
\label{eq:passages}
\begin{aligned}
	U_\varepsilon^+ &:= \{ (r,\theta) \in \R^2: r_1 \leq |r| \leq r_2,\, |\theta| < \varepsilon \},\\
	U_\varepsilon^- &:= \{ (r,\theta) \in \R^2: r_1 \leq |r| \leq r_2,\, |\theta-\pi| < \varepsilon \}.
\end{aligned}
\end{equation}
Then $\Omega_\varepsilon$ is a Lipschitz domain which satisfies Assumption~\ref{assumption}.
\begin{figure}[h]
\label{fig:counterexample}
\begin{center}
\begin{tikzpicture}[xscale=0.6,yscale=0.6]
\draw[thick] (0,0) circle [radius=3.4];
\draw[thick,domain=15:165] plot ({2.8*cos(\x)}, {2.8*sin(\x)});
\draw[thick,domain=195:345] plot ({2.8*cos(\x)}, {2.8*sin(\x)});
\draw[thick,domain=15:165] plot ({1.35*cos(\x)}, {1.35*sin(\x)});
\draw[thick,domain=195:345] plot ({1.35*cos(\x)}, {1.35*sin(\x)});
\draw[thick] (15:1.35) -- (15:2.8);
\draw[thick] (165:1.35) -- (165:2.8);
\draw[thick] (195:1.35) -- (195:2.8);
\draw[thick] (345:1.35) -- (345:2.8);
\draw[thick,dotted,domain=165:195] plot ({1.35*cos(\x)}, {1.35*sin(\x)});
\draw[thick,dotted,domain=165:195] plot ({2.8*cos(\x)}, {2.8*sin(\x)});
\draw[thick,dotted,domain=345:375] plot ({1.35*cos(\x)}, {1.35*sin(\x)});
\draw[thick,dotted,domain=345:375] plot ({2.8*cos(\x)}, {2.8*sin(\x)});
\fill (0,0) node {$B$};
\fill (-2,-0.07) node {$U_\varepsilon^-$};
\fill (2.05,-0.07) node {$U_\varepsilon^+$};
\fill (2.8,-2) node[anchor=west] {$A$};
\end{tikzpicture}
\caption{The domain $\Omega_\varepsilon = A \cup B \cup U_\varepsilon^+ \cup U_\varepsilon^-$.}
\end{center}
\end{figure}
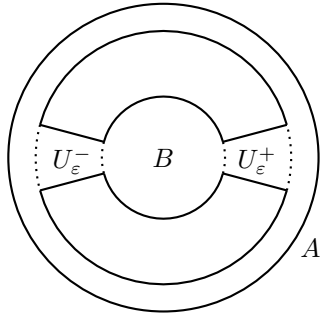
We claim that $\Omega_\varepsilon$ satisfies the claim of Theorem~\ref{thm:counterexample} for $\varepsilon>0$ small enough (which in practice may not be that small, but we will not attempt an explicit estimate). We divide the proof into the following three steps:

\emph{Step 1:} $\mu_2 (\Omega_\varepsilon) \to 0$ as $\varepsilon \to 0$.

\emph{Step 2:} If there is an antisymmetric eigenfunction associated with $\mu_2 (\Omega_\varepsilon)$, then there exists a (possibly different) eigenfunction $\psi_2$ whose nodal set contains either $\{x_1 = 0\} \cap \Omega_\varepsilon$ or $\{ x_2 = 0\} \cap \Omega_\varepsilon$. In particular, $\mu_2 (\Omega_\varepsilon) \geq \min \{\nu_1(\Omega_\varepsilon^{>}), \nu_1 (\Omega_\varepsilon^{\wedge}) \}$, where $\Omega_\varepsilon^{>} = \Omega_\varepsilon \cap \{x_1>0\}$ and $\Omega_\varepsilon^{\wedge} = \Omega_\varepsilon \cap \{x_2>0\}$, and $\nu_1$ is the first eigenvalue of the Laplacian with Dirichlet conditions on the relevant axis and Neumann conditions on the rest of the boundary. (See Figure~\ref{fig:omegaparts}.)\footnote{Actually, it is clear that we will have $\mu_2 (\Omega_\varepsilon) = \min \{\nu_1(\Omega_\varepsilon^{>}), \nu_1 (\Omega_\varepsilon^{\wedge}) \} = \nu_1(\Omega_\varepsilon^{>})$, and the nodal set of $\psi_2$ is exactly $\{ x_2 = 0\} \cap \Omega_\varepsilon$. But it appears to be more work to prove this than to deal with the additional case which comes from not proving it.}

\emph{Step 3:} $\min \{ \nu_1(\Omega_\varepsilon^{>}), \nu_1 (\Omega_\varepsilon^{\wedge}) \}$ is bounded from below away from $0$ as $\varepsilon \to 0$.\footnote{Although the statement of this step seems completely obvious, since we are dealing with a singular domain perturbation for the often delicate Neumann condition we include a proof.} Since $\mu_2 (\Omega_\varepsilon) \to 0$, this gives an immediate contradiction to the assumption that there was an antisymmetric eigenfunction associated with $\mu_2 (\Omega_\varepsilon)$ and thus proves the theorem.
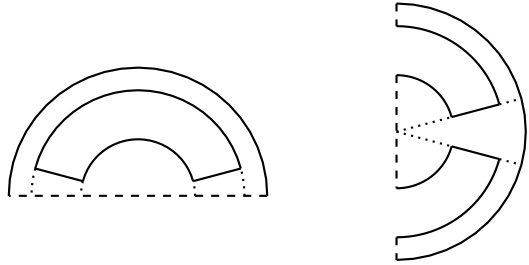
\begin{figure}[h]
\label{fig:omegaparts}
\begin{center}
\begin{tikzpicture}[xscale=0.5,yscale=0.5]
\draw[thick,domain=15:165] plot ({-3.4+2.8*cos(\x)},{2.8*sin(\x)});
\draw[thick,domain=15:165] plot ({-3.4+1.5*cos(\x)},{1.5*sin(\x)});
\draw[thick,dotted,domain=0:15] plot ({-3.4+2.8*cos(\x)},{2.8*sin(\x)});
\draw[thick,dotted,domain=0:15] plot ({-3.4+1.5*cos(\x)},{1.5*sin(\x)});
\draw[thick,dotted,domain=165:180] plot ({-3.4+2.8*cos(\x)},{2.8*sin(\x)});
\draw[thick,dotted,domain=165:180] plot ({-3.4+1.5*cos(\x)},{1.5*sin(\x)});
\draw[thick,domain=0:180] plot ({-3.4+3.4*cos(\x)},{3.4*sin(\x)});
\draw[thick,dashed] (0,0) -- (-6.8,0);
\begin{scope}[shift={(-3.4,0)}]
\draw[thick] (15:1.5) -- (15:2.8);
\draw[thick] (165:1.5) -- (165:2.8);
\end{scope}
\draw[thick,dashed] (3.4,5.1) -- (3.4,4.5);
\draw[thick,dashed] (3.4,3.2) -- (3.4,0.2);
\draw[thick,dashed] (3.4,-1.1) -- (3.4,-1.7);
\draw[thick,domain=270:450] plot ({3.4+3.4*cos(\x)},{1.7+3.4*sin(\x)});
\draw[thick,domain=270:345] plot ({3.4+2.8*cos(\x)},{1.7+2.8*sin(\x)});
\draw[thick,domain=15:90] plot ({3.4+2.8*cos(\x)},{1.7+2.8*sin(\x)});
\draw[thick,domain=270:345] plot ({3.4+1.5*cos(\x)},{1.7+1.5*sin(\x)});
\draw[thick,domain=15:90] plot ({3.4+1.5*cos(\x)},{1.7+1.5*sin(\x)});
\begin{scope}[shift={(3.4,1.7)}]
\draw[thick] (15:1.5) -- (15:2.8);
\draw[thick] (345:1.5) -- (345:2.8);
\draw[thick,dotted] (15:0) -- (15:1.5);
\draw[thick,dotted] (15:2.8) -- (15:3.4);
\draw[thick,dotted] (345:0) -- (345:1.5);
\draw[thick,dotted] (345:2.8) -- (345:3.4);
\end{scope}
\end{tikzpicture}
\caption{The domains $\Omega_\varepsilon^{\wedge}$ (left) and $\Omega_\varepsilon^{>}$ (right). The solid lines indicate Neumann boundary conditions, the dashed lines Dirichlet, and the dotted lines the additional Neumann conditions inserted in the proof of Step 3.}
\end{center}
\end{figure}

\emph{Proof of Step 1:} We construct a test function $\phi \in H^1(\Omega_\varepsilon)$ by setting
\begin{displaymath}
	\phi (x) = \begin{cases} -c_1 \qquad &\text{if } x \in B,\\ c_2 &\text{if } x \in A, \\
	-c_1 + \left(\frac{c_1+c_2}{r_2-r_1}\right)(r-r_1) \qquad &\text{if } 
	x = (r,\theta) \in U_\varepsilon^+ \cup U_\varepsilon^-,\end{cases}
\end{displaymath}
where $c_1 :=1/|B|$, $c_2 :=1/|A|$ are chosen so that $\int_{A \cup B} \phi = 0$, and the definition of $\phi$ on $U_\varepsilon^\pm$ says that $\phi$ interpolates linearly in the radial direction between $-c_1$ on $B$ and $c_2$ on $A$, meaning $\phi \in H^1 (\Omega_\varepsilon)$. We see that
\begin{displaymath}
	\int_{\Omega_\varepsilon} \phi^2 \geq \frac{1}{|A|}+\frac{1}{|B|},
\end{displaymath}
while since $c_1,c_2$ depend only on $r_1,r_2,r_3$ and $|U_\varepsilon^\pm| \sim \varepsilon$,
\begin{displaymath}
	\int_{\Omega_\varepsilon} \!|\nabla \phi|^2 = \int_{U_\varepsilon^+ \cup U_\varepsilon^-}\!
	\left| \frac{\partial\phi}{\partial r}\right|^2 \leq C(r_1,r_2,r_3)\varepsilon,\,
	\left|\int_{\Omega_\varepsilon}\phi\,\right| = \left|\int_{U_\varepsilon^+ \cup U_\varepsilon^-}\!\!\!\!\phi\,\right|
	\leq C(r_1,r_2,r_3)\varepsilon.
\end{displaymath}
It follows from the variational characterisation of $\mu_2 (\Omega_\varepsilon)$, using $\phi - \frac{1}{|\Omega_\varepsilon|} \int_{\Omega_\varepsilon} \phi$ as a test function, that (assuming without loss of generality that $|\Omega_\varepsilon|=1$)
\begin{displaymath}
	\mu_2 (\Omega_\varepsilon) \leq \frac{\int_{\Omega_\varepsilon} |\nabla \phi|^2}{\int_{\Omega_\varepsilon}
	\big(\phi - \int_{\Omega_\varepsilon} \phi \big)^2} =
	\frac{\int_{\Omega_\varepsilon} |\nabla \phi|^2}{\int_{\Omega_\varepsilon}\phi^2
	-\big(\int_{\Omega_\varepsilon}\phi\big)^2} \longrightarrow 0 \quad\text{as } \varepsilon \to 0.
\end{displaymath}

\emph{Proof of Step 2:} Since $\Omega_\varepsilon$ has reflection symmetry with respect to the axes $\{x_1=0\}$ and $\{x_2=0\}$, we may choose a (possibly different) basis of eigenfunctions 
for $\mu_2 (\Omega_\varepsilon)$ such that each is either symmetric or antisymmetric with respect to \emph{reflection} in each axis. We claim that there is at least one basis element which is antisymmetric in one axis: this will immediately imply the claim of the step.

Suppose not. Then every eigenfunction is symmetric in both axes (noting this property is preserved under taking linear combinations); in particular, this is true of our rotationally antisymmetric eigenfunction $\psi_2$, which satisfies $\psi_2(0)=0$. Since by the maximum principle $\psi_2$ cannot have an isolated zero, and indeed $\psi_2$ must change sign in every neighbourhood of every zero, there exists an open neighbourhood $U$ of $0$ such that $\{\psi_2>0\}$ and $\{\psi_2<0\}$ both have (at least) two connected components in $U$ (more precisely: on any circle $S_r = \{x: |x|=r\}$ for $r>0$ small enough, each will have at least two connected components at positive distance to each other). Since $\{\psi_2>0\}$ and $\{\psi_2<0\}$ are disjoint, open planar sets, it is impossible for them both to be connected. This contradicts Courant's theorem.

\emph{Proof of Step 3:} We first treat $\Omega_\varepsilon^{>}$. 
We decompose $\Omega_\varepsilon^{>}$ through the addition of Neumann conditions along $\{\theta = \pm \varepsilon \}$ as shown in Figure~\ref{fig:omegaparts} into two identical copies of a circular sector $S_\varepsilon$ of opening angle $\frac{\pi}{2}-\varepsilon$ and Dirichlet conditions on one side, two copies of a segment of annulus $A_\varepsilon$ of the same length of angle and Dirichlet conditions at one end, and a wedge $W_\varepsilon$ of opening angle $2\varepsilon$ and radial length $r_3$ and a Dirichlet condition at its vertex $0$ (and Neumann conditions elsewhere, in all cases). Then the variational characterisation immediately implies
\begin{displaymath}
	\nu_1 (\Omega_\varepsilon^{>}) \geq \min \{\nu_1 (A_\varepsilon), \nu_1 (S_\varepsilon), \nu_1 (W_\varepsilon) \},
\end{displaymath}
so it remains to bound the latter eigenvalues from below as $\varepsilon \to 0$. In fact we have
\begin{displaymath}
	\nu_1 (A_\varepsilon) \geq \nu_1 (A_0), \qquad \nu_1 (S_\varepsilon) \geq \nu_1 (S_0),
\end{displaymath}
where $A_0$ is the sector of angle $\frac{\pi}{2}$ and $S_0$ the corresponding quarter-annulus. These follow by a direct variational argument, since any test function on $A_\varepsilon$ or $S_\varepsilon$ which is zero on the Dirichlet boundary may be extended by zero across this boundary to obtain a valid test function on $A_0$ or $S_0$, respectively, with the same Rayleigh quotient.

Meanwhile, $\nu_1 (W_\varepsilon)$ is independent of $\varepsilon>0$; the corresponding eigenfunction, call it $\psi_\varepsilon$, depending only on the radial variable $r$ and not $\theta$ (this can be seen either by separating variables explicitly, or performing a ``symmetrisation'' by replacing $\psi_\varepsilon(r_0,\theta)$ by its mean value $\frac{1}{2\varepsilon}\int_{-\varepsilon}^\varepsilon \psi_\varepsilon (r_0,\theta)\,\textrm{d}\theta$ for each fixed $r_0>0$, which produces a new test function whose Rayleigh quotient cannot be larger than that of $\psi_\varepsilon$). 
Since $\nu_1 (W_\varepsilon)>0$, the eigenfunctions not being constant due to the condition $\psi_\varepsilon (0)=0$, this completes the proof for $\nu_1 (\Omega_\varepsilon^{>})$. 

We now treat $\nu_1 (\Omega_\varepsilon^{\wedge})$. In this case, we insert additional Neumann conditions along $\{(r,\theta): |r|=r_1,\, \theta \in (0,\varepsilon) \cup (\pi-\varepsilon,\pi) \}$ and $\{(r,\theta): |r|=r_2,\, \theta \in (0,\varepsilon) \cup (\pi-\varepsilon,\pi) \}$ to decompose $\Omega_\varepsilon^{\wedge}$ into the upper half disk $B^\wedge = B \cap \{x_2>0\}$, the upper half annulus $A^\wedge = A \cap \{x_2 > 0\}$ and two copies of the half passage $V_\varepsilon^+ = U_\varepsilon^+ \cap \{x_2 > 0\}$, cf.~Figure~\ref{fig:omegaparts}. Thus
\begin{displaymath}
	\nu_1 (\Omega_\varepsilon^{\wedge}) \geq \min \{ \nu_1 (A^\wedge), \nu_1 (B^\wedge), \nu_1 (V_\varepsilon^+) \},
\end{displaymath}
where, again, $\nu_1 (\,\cdot\,)$ is the first mixed Dirichlet-Neumann eigenvalue with Dirichlet conditions on $\{x_2=0\}$ and Neumann elsewhere. Then $\nu_1(A^\wedge),\nu_1(B^\wedge)>0$ are independent of $\varepsilon>0$, and a simple variational argument shows that $\nu_1 (V_\varepsilon^+)>0$ is increasing in $\varepsilon$: indeed, suppose $\varepsilon_1 < \varepsilon_2$. As for $A_\varepsilon$ and $S_\varepsilon$ above, by identifying $V_{\varepsilon_1}^+$ with a subset of $V_{\varepsilon_2}^+$ in the right way, we may extend any test function on the former by zero across its boundary to obtain (after rotation) a valid test function in $H^1(V_{\varepsilon_2}^+)$, still vanishing on the Dirichlet boundary of $V_{\varepsilon_2}^+$, with the same Rayleigh quotient. Thus the space of valid test functions on $V_{\varepsilon_1}^+$ may be identified with a subset of those on $V_{\varepsilon_2}^+$, immediately yielding $\nu_1 (V_{\varepsilon_1}^+) \geq \nu_1 (V_{\varepsilon_2}^+)$.

We conclude that $\liminf_{\varepsilon \to 0} \nu_1 (V_\varepsilon^+) = \lim_{\varepsilon \to 0}\nu_1 (V_\varepsilon^+) > 0$. This completes the proof of this step and hence the theorem.
\end{proof}

\bibliographystyle{amsplain}

\begin{thebibliography}{20}


\bibitem{abramowitz:72}
M.~Abramowitz and I.~Stegun, \emph{Handbook of mathematical functions with formulas, graphs, and mathematical tables}, 10th printing, National Bureau of Standards Applied Mathematics Series 55. For sale by the Superintendent of Documents, U.S. Government Printing Office, Washington, DC, 1972.

\bibitem{alessandrini:94}
G. Alessandrini, \emph{Nodal lines of eigenfunctions of the fixed membrane problem in general convex domains}, Comment. Math. Helvetici \textbf{69} (1994), 142--154.

\bibitem{fournais:01}
S.~Fournais, \emph{The nodal surface of the second eigenfunction of the Laplacian in ${\bf{R}}^D$ can be closed},  J. Differential Equations \textbf{173}  (2001), 145--159.

\bibitem{freitas:08}
P.~Freitas and D.~Krej{\v{c}}i{\v{r}}{\'{\i}}k, \emph{Location of the nodal set for thin curved tubes}, Indiana Univ. Math. J. \textbf{57} (2008), 343--375.

\bibitem{grebenkov:13}
D.S.~Grebenkov and B.-T.~Nguyen, \emph{Geometrical structure of Laplacian eigenfunctions}, SIAM Rev. \textbf{55} (2013), 601--667.

\bibitem{helffer:10}
B.~Helffer, \emph{On spectral minimal partitions: a survey}, Milan J. Math. \textbf{78} (2010), 575--590.

\bibitem{helffer:15}
B.~Helffer and T.~Hoffmann-Ostenhof, \emph{A review on large $k$ minimal spectral partitions and Pleijel's theorem}, Spectral theory and partial differential equations, 39--57, Contemp. Math. \textbf{640}, American Mathematical Society, Providence, RI, 2015.

\bibitem{hoffmann:97}
M.~Hoffmann-Ostenhof, T.~Hoffmann-Ostenhof and N.~Nadirashvili, \emph{The nodal line of the second eigenfunction of the Laplacian in ${\bf{R}}^2$ can be closed},  Duke Math.~J. \textbf{90}  (1997), 631--640.

\bibitem{kennedy:13}
J.B.~Kennedy, \emph{Closed nodal surfaces for simply connected domains in higher dimensions}, Indiana Univ. Math. J. \textbf{62} (2013), 785--798.

\bibitem{melas:92}
A.~D.~Melas, \emph{On the nodal line of the second eigenfunction of the Laplacian in $\R^2$}, J. Differential Geometry
\textbf{35} (1992), 255--263.

\bibitem{payne:67}
L.~E.~Payne, \emph{Isoperimetric inequalities and their applications}, SIAM Review \textbf{9} (1967), 453--488.

\bibitem{payne:73}
L.~E.~Payne, \emph{On two conjectures in the fixed membrane eigenvalue problem}, Z. Angew. Math. Phys. \textbf{24} (1973),
721--729.


\end{thebibliography}

\providecommand{\bysame}{\leavevmode\hbox to3em{\hrulefill}\thinspace}
\providecommand{\href}[2]{#2}

\end{document}